\documentclass{article}
\usepackage[utf8]{inputenc}
\usepackage[margin=0.5in]{geometry}
\usepackage{import}
\usepackage{graphicx}
\usepackage{caption}
\usepackage{subcaption}
\usepackage{enumerate}
\usepackage{hyperref}
\usepackage{tikz}
\usepackage{amssymb}
\usepackage{amsmath}
\usepackage{amsthm}
\usepackage{fp}	
\usepackage{xcolor,multirow}
\usepackage{enumitem}
\usepackage{xcolor}
\usepackage{hyperref}
\hypersetup{colorlinks=true, citecolor=black,linkcolor=blue, urlcolor=blue}
\newtheorem{thm}{Theorem}[section]

\newtheorem{cor}[thm]{Corollary}

\newtheorem{conj}[thm]{Conjecture}

\numberwithin{equation}{section}

\title{A Structural Approach to Burning Number}
\author{Jean Guillaume$^1$ and Tyriana Williams$^2$}

\begin{document}

\maketitle

\begin{abstract}
Graph burning is a deterministic discrete-time process that models the propagation of information within a network as a set of fires that spread in a graph. The associated graph parameter, the burning number, measures the speed of that spread. The smaller the burning number of a graph is, the faster information spreads in the corresponding network. In this paper, we mainly study the burning number from a structural point of view. In particular, we structurally characterize all graphs with burning number 2. We then proceed to compute the burning number of split graphs and $\{ K_2+2K_1, P_4, C_4\}$-free graphs. In particular, we show that connected split graphs and connected $\{ K_2+2K_1, P_4, C_4\}$-free graphs are well-burnable. Lastly, we provide a blueprint on how to investigate the burning number of a given social network. 
\end{abstract}

\section{Introduction}

Configurations of nodes and connections emerge naturally in various applications, including physical networks (electrical circuits, street networks, organic molecules, and so on) and social networks such as Facebook. Graph theory, the branch of mathematics that studies these types of configurations, formally known as \emph{graphs,} provides the fundamental framework that allows for the modeling and analysis of these networks. In return for its applications and implications on social network analysis, the mathematical product engendered by the K\"{o}nigsberg Bridge Problem (the problem that is often said to have been the birth of graph theory) has in recent years come out of the shadow of combinatorics to take a life of its own. Indeed, graph theory has become an invaluable contributor to understanding social networks which include groups of interacting individuals, organizations, or entities. Figure~\ref{Facebook} depicts a Facebook network comprising 50 users (vertices) and pairwise interactions (edges) amongst them.

\begin{figure}[h]
\centerline{\includegraphics[scale=.30]{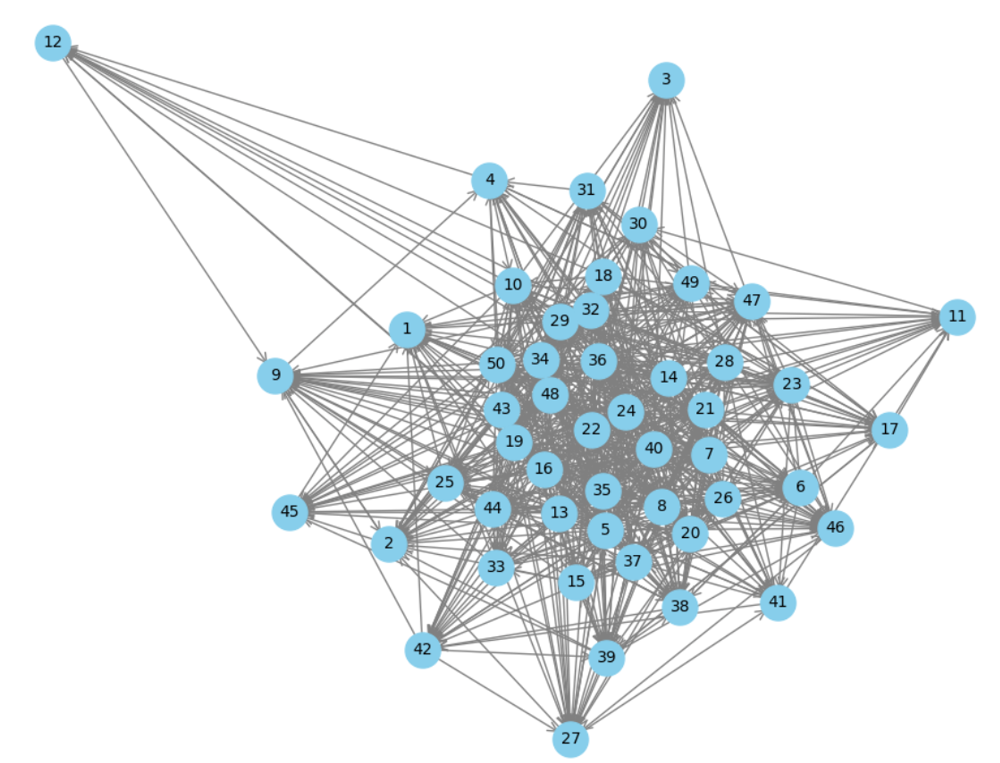}}
\caption{Small Facebook network of 50 people \label{Facebook}} 
\end{figure} 

Mainly due to its preeminent rise in popularity, social media has long played a dominant role in the spread of information and in the design of viral market strategies. In fact, one of the current active topics in social network analysis is the spread of social influence~\cite{Bonato1}. According to a recent study on the spread of emotional contagion in Facebook~\cite{Facebookstudy}, the underlying network plays an essential role as a medium for the spread of information, ideas, and influence among its members. In particular, in-person interaction and nonverbal cues are not strictly necessary for the spread. A specific example is the spread of online memes on social media platforms such as Instagram, Facebook, and X. Influence spreads from a node (an agent) to each of its neighbors (followers). Moreover, though the influence originates from a single source node, new sources    of influence emerge over time in various locations in the network. Operating under these assumptions, Bonato et al. recently proposed a graph theoretic process called \emph{graph burning} for analyzing the rate of the information spread over a network~\cite{Bonato1}. Note that the information spread is considered as a fire spread over the network. 

We now define the burning process on a \emph{finite, simple, undirected} graph. Initially, all the vertices of the given graph are unburned. We then proceed to burn all the vertices in discrete time-steps or rounds. In each round, not only is a new fire started by burning a new unburned node but also all vertices adjacent to a previously burned vertex (an existing fire) will start burning as well. If a node is burned, it remains in that state until the end of the process, which is when all nodes are burned. Note that for each round, an available unburned node must be chosen as a new source of fire. Two visual illustrations of the burning process are given in figure~\ref{Spider} using the same graph, known as a \emph{spider} in graph theory circles. The burning source chosen in the $i$-th round is denoted $x_i$ and the consecutively chosen sources of fire form what is called a \emph{burning sequence.}

\begin{figure}[h!]
    \centering 

    \begin{tikzpicture}
 
        \node (0) [draw, shape= circle, fill=black, scale=.4] at (0,0) {};
        \node (1) [draw, shape= circle, fill=black, scale=.4] at (0,2) {};
        \node (2) [draw, shape= circle, fill=black, scale=.4] at (2,0) {};
        \node (3) [draw, shape= circle, fill=black, scale=.4] at (-2,0) {};
        \node (4) [draw, shape= circle, fill=black, scale=.4] at (0,-2) {};
        \node (5) [draw, shape= circle, fill=black, scale=.4] at (2,-2) {};
        
        \draw (4)--(0) --(2)--(5);
        \draw (1) --(0)--(3);
        \draw (2, -2) circle (0.2);
        \node at (2.5,-2) {$x_1$};
        \draw (0, 0) circle (0.2);
        \node at (0.4,0.3) {$x_2$};

        \draw (0, 2) circle (0.2);
        \node at (0.5, 2) {$x_3$} ;
        \node at (0, -3) {burning process (a)} ;

        \node (a) [draw, shape= circle, fill=black, scale=.4] at (8,0) {};
        \node (b) [draw, shape= circle, fill=black, scale=.4] at (8,2) {};
        \node (c) [draw, shape= circle, fill=black, scale=.4] at (10,0) {};
        \node (d) [draw, shape= circle, fill=black, scale=.4] at (6,0) {};
        \node (e) [draw, shape= circle, fill=black, scale=.4] at (8,-2) {};
        \node (f) [draw, shape= circle, fill=black, scale=.4] at (10,-2) {};
        
        \draw (e)--(a) --(c)--(f);
        \draw (b) --(a)--(d);
        \draw (10, -2) circle (0.2);
        \node at (10.5,-2) {$x_2$};

        \draw (8, 0) circle (0.2);
        \node at ((8.5, 0.3) {$x_1$} ;
        \node at (8, -3) {burning process (b)} ;
    
     \end{tikzpicture}
    \caption{Illustration of the burning process}
    \label{Spider}
\end{figure}

These two illustrations highlight the fact that the burning process may have a lot to do with the choice of the burning sources. For example, at the end of the second round of the burning process (a), there are still three remaining nodes left to be burned, namely the pendant vertices not labeled $x_1.$ Since a new fire must be started in each round, we choose one of these three unburned vertices to be $x_3.$ This leads to the length of the burning sequence in (a) to be 3 while the length in (b) is 2. In fact, depending on the round in which a vertex is selected as a source, the process may be shorter or longer. Therefore, if the goal is to complete the process in the least number of rounds possible, resulting in an optimal burning sequence, then these burning sources must be strategically chosen. The \emph{burning number} of a graph $G,$ denoted $b(G),$ is the minimum number of steps required to burn all the vertices of $G.$ In other words, the burning number $b(G)$ is the length of an optimal burning sequence. Equivalently, the burning number measures the optimal speed at which a graph can be burned. The lower the burning number is, the faster the fire spreads. Also, it follows immediately from the definition that for any graph with two or more vertices, the burning number is greater or equal to 2.
The illustration on the left, labeled burning process (a),  shows a burning process comprising three sources fire. It is worth noting  Observe that the length of the sequence in the burning process (a) is 3 while the length in (b) is 2. In fact, depending on the round in which a vertex is selected as a source, the process may be shorter or longer. Therefore, if the goal is to complete the process in the least number of rounds possible, resulting in an optimal burning sequence, then these burning sources must be strategically chosen. The \emph{burning number} of a graph $G,$ denoted $b(G),$ is the minimum number of steps required to burn all the vertices of $G.$ In other words, the burning number $b(G)$ is the length of an optimal burning sequence. Equivalently, the burning number measures the optimal speed at which a graph can be burned. The lower the burning number is, the faster the fire spreads. Also, it follows immediately from the definition that for any graph with two or more vertices, the burning number is greater or equal to 2.

It is worth noting that a similar process on $n$-dimensional hypercubes first appeared in~\cite{Transmission}. The motivation then was a transmission problem from Intel, formulated by Aron, and they proved that the burning number of an $n$-dimensional hypercube is $\lceil \frac{n}{2}\rceil +1.$ This process was later independently defined by Bonato et al. in~\cite{Bonato1} and, this time, the motivation comes from the preeminent role that social networks play in spreading information (news, memes, opinions, trends, etc.) Since then, the topic of graph burning has garnered a great deal of attention, giving rise to a significant body of work. Bonato put the count at over two dozen papers in~\cite{Survey}, dedicated mainly to bounding the number either in general or  in the case of specific graph classes and on the computational complexity of the graph burning process. Note that the process is even \textbf{NP}-complete for fairly restrictive graph classes such as trees of maximum degree 3, according to~\cite{Bonato2}. 

In this paper, we consider a somewhat different approach to the burning number problem. More specifically, we seek to determine the kinds of graph structure(s) that are allowed under a specific burning number. To this aim, we were able to give in section 3 a characterization of all graphs with burning number 2. We then proceed to show that \emph{connected split} graphs are \emph{well-burnable} as well as \emph{connected} $\{ K_2+2K_1, P_4, C_4\}$-free graphs. That is, the \emph{burning number conjecture} holds for these types of graphs. Lastly, in section 4, we highlight the work of an undergraduate research mentee, Tyriana Williams, in her effort to burn social media networks.  

Before we can achieve these objectives, we first need some definitions and results. Also, in what follows, we consider all graphs to be non-trivial, simple, finite, and undirected, unless stated otherwise.

\section{Preliminaries}

In this section, we present relevant definitions, related results, and the main conjecture as well as introduce some necessary notation and terminology. 

We begin with some basic definitions and terminology from graph theory. A $u, v$-path is a trail in a graph which does not repeat any vertex and whose endpoints are vertices $u$ and $v$. The length of a path is its number of edges. The distance $d(u, v)$ from vertex u to vertex $v$ in $G$ is the length of a shortest path from $u$ to $v$. The eccentricity of a vertex $v$ is defined as $\max\{ d(v, u) : u \in V(G) \}$, where $V(G)$ denotes the set of vertices of $G$. The radius of a graph $G$ is the minimum eccentricity over the set of all nodes in $G$. The center of $G$ consists of all nodes in $G$ with minimum eccentricity. The diameter of a graph $G$ is the maximum eccentricity over the set of all nodes in $G.$ In other words, the diameter is the maximum distance between any two vertices. Also, it is not hard to see that small diameters give rise to small burning numbers. Lastly, a subset of edges that connects all vertices of $G$ without forming any cycle (a closed path on at least three vertices) forms what is called a \emph{spanning tree} of $G.$ 

Bessy et al. claimed in~\cite{Bonato2} that ``burning a graph is hard." Nevertheless, they were able to bequeath us some of the key results on the topic of graph burning, namely the following in~\cite{Bonato1}:
\begin{itemize}
    \item[1.] \emph{Paths} play an important role in graph burning. 

    \begin{thm} \label{Pathresult} \emph{(\cite{Bonato1})} Let $G$ be a path $P_n$ or a cycle $C_n$ on $n$ vertices, we have that $b(G)= \lceil\sqrt{n} \rceil.$
    \end{thm}
    \item[2.] The burning number of a graph is related to other graph parameters such as its \emph{radius} and \emph{diameter.}
    
    \begin{thm}\label{bounds} \emph{(\cite{Bonato1})} For any graph $G$ with radius $r$ and diameter $d,$ we have that 
     $$\lceil (d+1)^{\frac{1}{2}} \rceil \leq b(G) \leq r+1$$
     \end{thm}
    
    \item[3.] The strategy to burn the \emph{spanning tree} of a graph $G$ still works when the remaining edges in $E(G)$ are added (i.e. \emph{Tree Reduction Theorem}).

    \begin{thm}\label{Reduction} \emph{(Tree Reduction Theorem,} \emph{\cite{Bonato1})} For a graph $G,$ we have that 
    $$b(G)= \emph{min}\{ b(T): T \emph{ is a spanning tree of } G\} $$ 
    \end{thm}

    \item[4.] The relatively foremost open problem in the area of graph burning is known as the \emph{burning number conjecture}, which states that : 
    
    \begin{conj}\label{Conjecture} \emph{(\cite{Bonato1})} If G is a connected graph of order n, then $b(G) \leq \lceil \sqrt{n} \rceil. $
        \end{conj}
   
\end{itemize}

 Remark that, according to theorem~\ref{Reduction}, proving the conjecture for trees is sufficient to prove it for any graph $G.$ Also, if the conjecture holds, paths are then confirmed to be among connected graphs on $n$ vertices with the largest burning number. Lastly, as a reminder, if the conjecture holds for a graph $G$ on $n$ vertices, the graph $G$ is said to be well-burnable. 

While the full conjecture has so far remained unresolved, it is however known to hold for special classes of graphs. For example, the conjecture is proved for various subclasses of trees, including caterpillars, trees with at most one vertex of degree at least 3, and trees without degree 2 vertices~\cite{BoundsandHardness}; so is any graph with a \emph{Hamiltonian} path or cycle, which is a graph with a path or cycle that goes through each node exactly once. In such a case where the conjecture holds, we say the the corresponding graph is \emph{well-burnable.} In fact, as previously stated, one of the main results in this paper shows that connected \emph{split graphs} and connected $\{ K_2+2K_1, P_4, C_4\}$-free graphs are well-burnable.

A \emph{split} graph is one whose vertex set can be partitioned into a \emph{clique} and an \emph{independent} set of vertices~\cite{Dominanceorder}. That is, a subset of mutually adjacent vertices and a subset of mutually non-adjacent vertices, respectively. Some proper subfamilies of split graphs are well-known for their remarkable properties. For example, the \emph{threshold graphs}, first defined by Chv\'{a}tal and Hammer in~\cite{Chvatal77}, are split graphs that do not contain $2K_2, C_4,$ and $P_4$ as \emph{induced subgraphs.} That is, the process of vertex deletion on a threshold graph cannot produce any of the 3 graphs listed in figure~\ref{Threshold}: two independent edges, the cycle on four vertices, and the path on four vertices. Technically speaking, we say that threshold graphs are $\{2K_2, C_4, P_4\}$-free or the graphs listed in figure~\ref{Threshold} are forbidden subgraphs of threshold graphs. Note that threshold graphs can be constructed from a single vertex by repeatedly adding either an \emph{isolated} vertex (a vertex that is independent of all previous vertices) or a \emph{dominating} vertex (a vertex that is adjacent to all previous vertices). In general, split graphs are known to be $\{2K_2, C_4, C_5\}$-free, where $C_5$ denotes the cycle on five vertices.   

\bigskip

\begin{figure}[h!]
    \centering 

    \begin{tikzpicture}
 
        \node (0) [draw, shape= circle, fill=black, scale=.4] at (0,0) {};
        \node (1) [draw, shape= circle, fill=black, scale=.4] at (0,2) {};
        \node (2) [draw, shape= circle, fill=black, scale=.4] at (2,0) {};
        \node (3) [draw, shape= circle, fill=black, scale=.4] at (2, 2) {};

        \draw (0)--(1);
        \draw (2) --(3);

        \node (a) [draw, shape= circle, fill=black, scale=.4] at (5,0) {};
        \node (b) [draw, shape= circle, fill=black, scale=.4] at (5,2) {};
        \node (c) [draw, shape= circle, fill=black, scale=.4] at (7,0) {};
        \node (d) [draw, shape= circle, fill=black, scale=.4] at (7, 2) {};

        \draw (a)--(b)--(d)--(c)--(a);
    
       \node (e) [draw, shape= circle, fill=black, scale=.4] at (10,0) {};
        \node (f) [draw, shape= circle, fill=black, scale=.4] at (10,2) {};
        \node (g) [draw, shape= circle, fill=black, scale=.4] at (12,0) {};
        \node (h) [draw, shape= circle, fill=black, scale=.4] at (12, 2) {};

        \draw (e)--(g)--(h)--(f);

        \node at (1.2,-0.5) {$2K_2$} ;
        \node at (6.2,-0.5) {$C_4$} ;
        \node at (11.2,-0.5) {$P_4$} ;

     \end{tikzpicture}
    \caption{Forbidden subgraphs of threshold graphs}
    \label{Threshold}
\end{figure}

Also of interest to us is another class of graphs that is known to be $\{K_2+2K_1, C_4, P_4\}$-free~\cite{Dominanceorder}. In other words, graphs that do not contain any of the graphs listed in figure~\ref{Jean} as induced subgraphs. They are known in \cite{Dominanceorder} as \emph{threshold covered} graphs. In particular, we show that any connected $\{ K_2+2K_1, P_4, C_4\}$-free graph is well-burnable. Note that $\{ K_2+2K_1, P_4, C_4\}$-free graphs can be constructed sequentially from a $2K_2$ by repeatedly adding as many \emph{dominating} vertices as needed.

\bigskip

\begin{figure}[h!]
    \centering 

    \begin{tikzpicture}
 
        \node (0) [draw, shape= circle, fill=black, scale=.4] at (0,0) {};
        \node (1) [draw, shape= circle, fill=black, scale=.4] at (0,2) {};
        \node (2) [draw, shape= circle, fill=black, scale=.4] at (2,0) {};
        \node (3) [draw, shape= circle, fill=black, scale=.4] at (2, 2) {};

        \draw (0)--(1);

        \node (a) [draw, shape= circle, fill=black, scale=.4] at (5,0) {};
        \node (b) [draw, shape= circle, fill=black, scale=.4] at (5,2) {};
        \node (c) [draw, shape= circle, fill=black, scale=.4] at (7,0) {};
        \node (d) [draw, shape= circle, fill=black, scale=.4] at (7, 2) {};

        \draw (a)--(b)--(d)--(c)--(a);
    
       \node (e) [draw, shape= circle, fill=black, scale=.4] at (10,0) {};
        \node (f) [draw, shape= circle, fill=black, scale=.4] at (10,2) {};
        \node (g) [draw, shape= circle, fill=black, scale=.4] at (12,0) {};
        \node (h) [draw, shape= circle, fill=black, scale=.4] at (12, 2) {};

        \draw (e)--(g)--(h)--(f);

        \node at (1.2,-0.5) {$K_2+2K_1$} ;
        \node at (6.2,-0.5) {$C_4$} ;
        \node at (11.2,-0.5) {$P_4$} ;

     \end{tikzpicture}
    \caption{Forbidden subgraphs in $\{K_2+2K_1, C_4, P_4\}$-free graphs}
    \label{Jean}
\end{figure}

We close this section by highlighting a graph parameter known to have an inverse relationship with its diameter. Commonly referred to as the graph \emph{density}, it is a measure of how close a graph is to being \emph{complete}, which is to being a graph in which every pair of distinct vertices is connected. Formally speaking, it is the ratio of the number of existing edges to the total number of possible edges~\cite{Problem}. Thence, we have the following duality relation between a graph density and its diameter: higher density leads to lower graph diameter and vice versa.

We now proceed as follows. In the next section, we give a degree sequence characterization of all graphs with burning number 2. Then we proceed to compute the burning number of the entire class of split graphs as well as $\{ K_2+2K_1, P_4, C_4\}$-free graphs, which are known to be non-splits since they contain $2K_2.$ Lastly, in section 4, we highlight the work of an undergraduate research mentee, Tyriana Williams, in her effort to burn social media networks.

\section{Graph characterization \& burning number of graph classes}

We begin this section by establishing a necessary and sufficient condition that any graph $G$ with $b(G)=2$ must satisfy. 

\subsection{A necessary and sufficient condition}

\begin{thm}\label{Characterization} For any graph $G,$ we have that $b(G) = 2$ if and only if $\Delta(G) \geq n-2.$ That is, if and only if the maximum degree of $G$ is at least $n-2.$ \end{thm}

\begin{proof} 

    \textbf{Necessity:}  Let $G$ be a graph such that $b(G)=2.$ Thus, by definition, there exists an optimal burning sequence $(x_1, x_2).$ Therefore, either $x_1$ is a dominating vertex or $x_1$ is adjacent to all vertices of $G$ except $x_2.$ That is, the degree of $x_1$ is either $n-1$ or $n-2.$ 

    \textbf{Sufficiency:} Now, let $G$ be a graph such that $\Delta(G) \geq n-2.$ That is, there exist $v \in V(G)$ such that  the degree of $v$ is at least $n-2.$ Suppose there is a vertex $v$ of degree $n-1.$ Therefore, $v$ is a dominating vertex. As a result, if we choose $v$ to be the first source of fire, then the entire graph will be burned on the second round of the burning process. Now suppose that $v$ is a max degree vertex of degree $n-2.$  That is, $v$ is adjacent to all vertices of $G$ but one. Let $v'$ be the vertex that is non-adjacent to $v.$ Therefore, $(v, v')$ is an optimal burning sequence of $G.$

    Hence the desired result.
\end{proof}

We now move on to compute the burning number of split graphs. In particular, we show that every connected split graph is well-burnable. 

\subsection{Burning number of split graphs}

Recall that a split graph is one whose vertex set can be partitioned into a \emph{clique} and an \emph{independent} set of vertices~\cite{Dominanceorder}. To simplify matter in what follows, we let $G = (C \cup I, E),$ where $E$ stands for the set of edges, $C$ is a maximal clique of $G,$ and $I$ the independent set formed by $V(G)\backslash C.$

We first show that any connected split graph is well-burnable.

\begin{thm}\label{Connectedsplit}
    Let $G=(C \cup I, E)$ be a connected split graph. Then $b(G) \leq 3.$ In particular, $b(G) =3$ only if $\Delta(G) < n-2.$
\end{thm}
\begin{proof}
    Let $G=(C\cup I, E)$ be a connected split graph. If $I$ is empty, then the desired result follows from theorem~\ref{Characterization}. Suppose now that $I$ is not empty. Then every vertex in $I$ is connected to a vertex in $C.$ Therefore, the radius of $G$ is at most 2. Therefore, the result follows from theorems~\ref{bounds} \&~\ref{Characterization}.
\end{proof}

Consequently, connected split graphs are well-burnable.
\begin{cor}
    Let $G$ be a connected threshold graph. Then $G$ is well-burnable.
\end{cor}

\begin{proof}
    It follows from the fact that threshold graphs are split.
\end{proof}

We now proceed to compute the burning number of disconnected split graphs. First, it is not hard to see that, based on the structural definition of split graphs, disconnected split graphs are limited to at most one non-trivial connected split component. That is, a disconnected split graph is either a set of isolated vertices or a connected split component together with a set of isolated vertices. 

\begin{thm}
    Let $G$ be a disconnected split graph. 
    \begin{itemize}

    \item[Case 1:] If $G$ does not include a connected component, then $b(G) = n,$ where $n \in \mathbb{N}$ is the number of isolated vertices. 
    
    \item[Case 2:] If $G$ comprises a connected component, denoted $G_c,$ and a set of $n$ isolated vertices, then 
    \[ b(G)= \begin{cases} \leq 3, & \,\,  \emph{ if } n \leq 2 \\ n+1, & \,\, \emph{otherwise} \end{cases} \]  
    \end{itemize}
    
\end{thm}

\begin{proof} To prove case 1, it suffices to realize that every isolated vertex is its own source of fire. When it comes to case 2, since every isolated vertex is its own source of fire, the optimal strategy is to choose the first source of fire to be a vertex in the connected component and the other sources of fire to be isolated vertices. Then, the result follows from theorem~\ref{Connectedsplit} and since every isolated vertex is its own source of fire.

\end{proof}

\subsection{Burning number of \texorpdfstring{$\{ K_2+2K_1, P_4, C_4\}\text{-free}$}{threshold covered} graphs}

Lastly, we show that connected $\{ K_2+2K_1, P_4, C_4\}$-free graphs are well-burnable. 

\begin{thm}
    Let $G$ be a $\{ K_2+2K_1, P_4, C_4\}$-free graph. Then $b(G) \leq 3.$ Moreover, $b(G) = 3$ only if $G$ is a $2K_2,$ which is a disconnected graph.
    \end{thm}

    \begin{proof}
        Let $G$ be a connected $\{ K_2+2K_1, P_4, C_4\}$-free graph. Thus, $G$ has at least a dominating vertex. Then  by theorem~\ref{Characterization}, $b(G) = 2.$ Otherwise, $G$ comprises two independent edges and $b(G)= 3.$ 
    \end{proof}

Consequently, connected $\{ K_2+2K_1, P_4, C_4\}$-free graphs are well-burnable.

\section{Burning social media networks}

This section is dedicated to showcase some of the work performed by my undergraduate research mentee, Tyriana Williams, in the summer of 2025. The main objective was to come up with a blueprint on how to burn social media networks. 

Since social network data of large platforms such as Facebook or Instagram is not available to the general public due to privacy policies, a Python program was created to randomly generate graphs that simulate the aforementioned networks. When the program runs, the user is prompted to enter how many users they want in their social network. o access the Jupyter Notebook code, visit the GitHub repository\footnote{\url{https://github.com/JeanGuil/social-networks-graph.git}}.

To evaluate the structural properties of randomly generated social media networks, the Python program was executed several times, each time with an increasing number of users in the network: 5, 10, 25, and 50. For each run, the resulting graph was visualized and analyzed. Key graph-theoretic parameters were computed, including the number of vertices, number of edges, diameter, radius, and center of the graph. The following are four distinct simulations.  

\subsection{5-user social media network simulation}

The simulation below corresponds to a sparse network of 5 users. The resulting graph is a low density one. Using brute force, it can be determined that the \textit{burning number} of the graph is 3, which can be verified using both theorems~\ref{bounds} and~\ref{Pathresult}. More specifically, we have that:
$$\lceil \sqrt{4 + 1} \rceil \leq b(G) \leq 2 + 1   \,\, \mbox{ and } \,\, b(G) = \lceil \sqrt{5} \rceil $$

\begin{figure}[h]

\begin{subfigure}{0.5\textwidth}

\centering
\includegraphics[width=0.6\linewidth, height=5cm, scale=3]{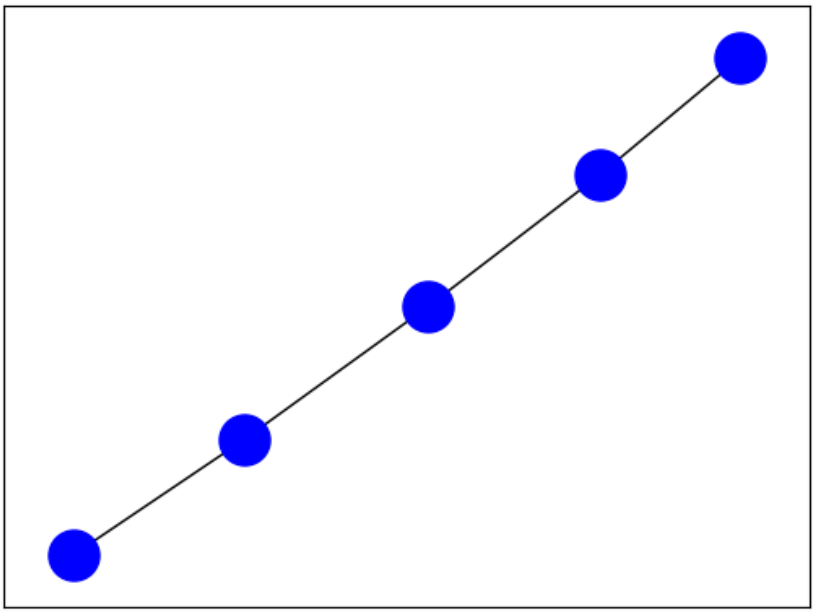} 
\end{subfigure}
{\hskip0.5cm} 
\begin{subfigure}{0.5\textwidth}
\begin{tabular}{|c|c|c|c|} \hline $n$ & Rad & Diam & $b(G)$ \\ \hline 5 & 2 & 4 & 3 \\ \hline  \end{tabular}
\bigskip
\bigskip
\bigskip
\end{subfigure}

\caption{Simulation of a 5-user social media network }
\label{fig:image2}
\end{figure}

\subsection{10-user social media network simulation}

The simulated network below comprises 10 users. Visibly, it is less sparse and has a higher density than the previous one. Again, using brute force, it can be determined that the \textit{burning number} of the graph is 3, which can be verified using theorem~\ref{bounds}. More specifically, we have that:
$$\lceil \sqrt{4 + 1} \rceil \leq b(G) \leq 2 + 1$$

\begin{figure}[h]

\begin{subfigure}{0.5\textwidth}

\centering
\includegraphics[width=0.6\linewidth, height=5cm, scale=3]{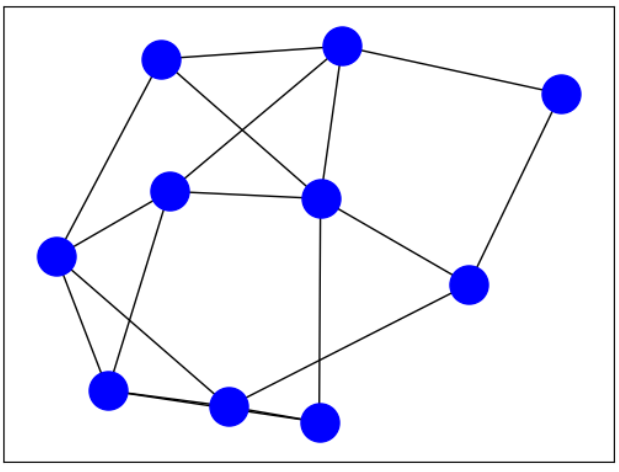} 
\end{subfigure}
{\hskip0.5cm} 
\begin{subfigure}{0.5\textwidth}
\begin{tabular}{|c|c|c|c|} \hline $n$ & Rad & Diam & $b(G)$ \\ \hline 10 & 2 & 3 & 3 \\ \hline  \end{tabular}
\bigskip
\bigskip
\bigskip
\end{subfigure}

\caption{Simulation of a 10-user social media network }
\label{fig:image2}
\end{figure}

\subsection{25-user social media network simulation}

This time, the simulated network encompasses 25 users. The resulting graph is one with a higher density than the two previous ones. The diameter is numerically estimated at 3 and radius at 2, which leads to the burning number being less or equal to 3 by theorem~\ref{bounds}.
$$\lceil \sqrt{3+ 1} \rceil \leq b(G) \leq 2 + 1$$

\begin{figure}[h]

\begin{subfigure}{0.5\textwidth}

\centering
\includegraphics[width=0.6\linewidth, height=5cm, scale=3]{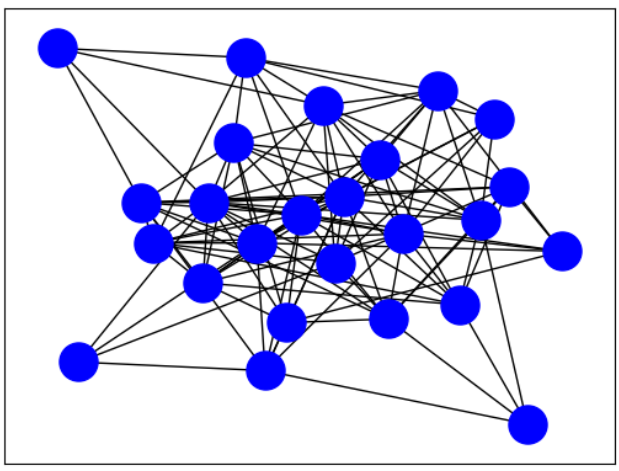} 
\end{subfigure}
{\hskip0.5cm} 
\begin{subfigure}{0.5\textwidth}
\begin{tabular}{|c|c|c|c|} \hline $n$ & Rad & Diam & $b(G)$ \\ \hline 25 & 2 & 3 & $\leq 3$\\ \hline  \end{tabular}
\bigskip
\bigskip
\bigskip
\end{subfigure}

\caption{Simulation of a 25-user social media network }
\label{fig:image2}
\end{figure}

\subsection{50-user social media network simulation}

This final simulation corresponds to a network of 50 users. It produces a high density graph which leads to a smaller diameter and radius. Intuitively, we expect a low burning number due to the inverse relationship between density and diameter recorded in real-world networks. This is known as the \emph{small-world} phenomenon where any two entities in the network are likely to be connected through a short sequence of intermediate acquaintances~\cite{Smallworld}. Indeed, using theorem~\ref{bounds}, we have that $b(G) =2.$ 
$$\lceil \sqrt{2+ 1} \rceil \leq b(G) \leq 1 + 1$$

\begin{figure}[h]

\begin{subfigure}{0.5\textwidth}

\centering
\includegraphics[width=0.6\linewidth, height=5cm, scale=3]{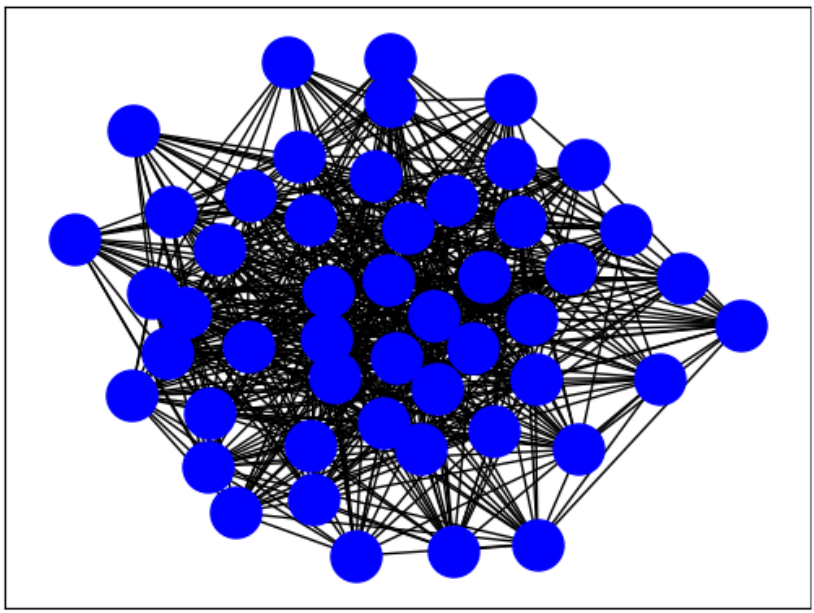} 
\end{subfigure}
{\hskip0.5cm} 
\begin{subfigure}{0.5\textwidth}
\begin{tabular}{|c|c|c|c|} \hline $n$ & Rad & Diam & $b(G)$ \\ \hline 50 & 1 & 2 & 2 \\ \hline  \end{tabular}
\bigskip
\bigskip
\bigskip
\end{subfigure}

\caption{Simulation of a 50-user social media network }
\label{fig:image2}
\end{figure}

\bigskip
 In conclusion, when it comes to understanding how fast information or influence spreads within a network, we have these two main takeaways from the above simulations: 1) higher density corresponds to smaller diameter and radius, which translates into smaller burning number by theorem~\ref{bounds} and 2) an increase in connectivity translates into a decrease in burning number, hence theorem~\ref{Reduction}. Takeaway 1 confirms what should be an intuitive fact: highly connected social networks facilitate faster information diffusion. Equivalently, information spreads quickly and efficiently in close-knit, cohesive groups. Takeaway 2 tells however a deeper story, which is the strategy to burn a graph still works when edges are added.

\section{Conclusion and future directions}

In this paper, we give a necessary and sufficient condition that all graphs with burning number equal to 2 must satisfy. More specifically, we prove that, for any graph of order $n,$ we have $b(G) = 2$ if and only if $G$ has a vertex of degree at least $n-2.$ We successfully compute the burning number for the entire class of split graphs as well as $\{ K_2+2K_1, P_4, C_4\}$-free graphs, which are known to be non-split since they contain $2K_2.$ Moreover, we prove that connected split graphs and connected $\{ K_2+2K_1, P_4, C_4\}$-free graphs are well-burnable. Lastly, the inverse relationship between a graph connectivity and its burning number is illustrated through simulations. 

Naturally, the next logical step is to characterize all graphs with burning number greater than 2. Being able to do so will also bring full resolution to long-standing related problems such as the characterization of graphs with prescribed radius, diameter, and center. To this aim, we must first remind ourselves of Bonato's warning that "burning a graph is hard"~\cite{Bonato2}, let alone characterize all graphs of a specific burning number. By the same token, we, mathematicians, don't choose to do something because it is easy. Here we go.

\section{Acknowledgments}
The realization of this project would not be possible without the critical support of the Provost Tenure-Readiness Research program @ Norfolk State University. Under this initiative and the leadership of Dr. Khadijah O. Miller, selected tenure-track faculty such as Dr. Jean Guillaume were provided with a generous summer stipend to engage in summer research along with the opportunity to mentor undergraduate student(s). Also, the 2026 Faculty Writing Network launched by the Deans from all  NSU Schools and Colleges together with Center for Teaching and Learning provided the extra writing motivation needed to complete the project. Last but not least, special thanks to my fellow node members of the writing network, Dr. Jocelyn Heath and Dr. Jessica M. Woltz, for their constructive feedback.

\vspace{3cm}


\end{document}